\documentclass[a4paper,12pt,reqno]{amsart}
\usepackage{latexsym}
\usepackage{amssymb,amsfonts,amsmath,mathrsfs,graphicx,multirow,caption}
\usepackage{color,dsfont}
\usepackage{xcolor}
\newtheorem{theorem}{Theorem}
\newtheorem{proposition}{Proposition}

\newtheorem{corollary}{Corollary}
\theoremstyle{definition}

\begin{document}

\title[Zeros of the derivatives of Hardy's function $Z(t)$]{A note on the zeros of the derivatives of Hardy's function $Z(t)$}
\author{Hung M. Bui and R. R. Hall}
\address{Department of Mathematics, University of Manchester, Manchester M13 9PL, UK}
\email{hung.bui@manchester.ac.uk}
\address{Department of Mathematics, University of York, York YO10 5DD, UK}
\email{richardroxbyhall@gmail.com}

\begin{abstract}
Using the twisted fourth moment of the Riemann zeta-function we study large gaps between consecutive zeros of the derivatives of Hardy's function $Z(t)$, improving upon previous results of Conrey and Ghosh [J. London Math. Soc. 32 (1985), 193--202], and of the second named author [Acta Arith. 111 (2004), 125--140]. We also exhibit small distances between the zeros of $Z(t)$ and the zeros of $Z^{(2k)}(t)$ for every $k\in\mathbb{N}$, in support of our numerical observation that the zeros of $Z^{(k)}(t)$ and $Z^{(\ell)}(t)$, when $k$ and $\ell$ have the same parity, seem to come in pairs which are very close to each other. The latter result is obtained using the mollified discrete second moment of the Riemann zeta-function.
\end{abstract}

\allowdisplaybreaks

\maketitle

\section{Introduction}

Hardy's function $Z(t)$ is defined by the equation
\begin{align*}
Z(t)&:=e^{i\theta(t)}\zeta(\tfrac12+it)\\
&=\bigg(\pi^{-it}\frac{\Gamma(\frac14+\frac{it}{2})}{\Gamma(\frac14-\frac{it}{2})}\bigg)^{1/2}\zeta(\tfrac12+it).
\end{align*}
Let $Z^{(k)}(t)$ be the $k$-th derivative of $Z(t)$ and $S^{(k)}=\{t_n^{(k)}\}_{n\in\mathbb{N}}$ denote the non-decreasing sequence of non-negative zeros of $Z^{(k)}(t)$, counted according to multiplicity. Let
\[
\Lambda^{(k)}:=\limsup_{n\rightarrow\infty}\frac{t_{n+1}^{(k)}-t_n^{(k)}}{2\pi/\log t_n^{(k)}}.
\]
Note that the gap between consecutive zeros $t_n^{(k)}$ and $t_{n+1}^{(k)}$ is normalised by the factor $\log t_n^{(k)}/2\pi$ so that the average spacing is $1$ (if we assume the Riemann Hypothesis (RH)). For all $k$ we have
\[
\Lambda^{(k)} \geq A(k): = \sqrt{\frac{2k+3}{2k+1}}.
\]
This follows from Wirtinger's Inequality (as stated in Theorem \ref{Wi} below) together with [\textbf{\ref{H2}}; Theorem 3]. In the case $k=1$, Conrey and Ghosh [\textbf{\ref{CG}}] showed that $\Lambda^{(1)}>1.4$ conditional on RH, and the second named author obtained $\Lambda^{(1)}>1.5462$ unconditionally [\textbf{\ref{H}}; Theorem 4]. 
We improve these results in the following theorem.

\begin{theorem}\label{thm1}
We have
\[
\Lambda^{(1)}>1.9.
\]
\end{theorem}

Beside the fourth moments of $Z(t)$ and its derivative, Theorem 4 of [\textbf{\ref{H}}] relies on a Wirtinger type inequality of the form
\begin{equation*}
	\frac{3}{8}\big(1+4\nu+\sqrt{1+8\nu}\big)\int_a^b|f(t)|^4dt\leq \Big(\frac{b-a}{\pi}\Big)^4	\int_{a}^{b}|f'(t)|^4dt+6\nu\Big(\frac{b-a}{\pi}\Big)^2	\int_{a}^{b}|f(t)f'(t)|^2dt
\end{equation*}
for any $f\in C^2[a,b]$ satisfying $f(a)=f(b)=0$ and for any $\nu\geq0$ [\textbf{\ref{H1}}; Theorem 2]. The details of this proof are contained in [\textbf{\ref{H}}; pp. 138--139]. It is easy to see how this may be applied to higher derivatives, given that the corresponding moments (and mixed moments) may be computed, and we thereby obtain
\[
\Lambda^{(k)} > B(k)
\]
with $B(k)$ given below. 
\vspace{2mm}

\begin{table}[h!]
	\centering
	\begin{tabular}{ |c|c||c|c||c|c|}
		\hline
		$k$ & $B(k)$ &  $k$ & $B(k)$ & $k$ & $B(k)$ \\
		\hline\hline
		$1$ & $1.5462$ & $6$ & $1.1479$ & $14$ & $1.0677$ \\[3pt]
		$2$ & $1.3609$ & $7$ & $1.1288$ & $18$ & $1.0533$ \\[3pt]
		$3$ & $1.2653$ & $8$ & $1.1141$ & $22$ & $1.0439$ \\[3pt]
		$4$ & $1.2099$ & $9$ & $1.1024$ & $26$ & $1.0373$ \\[3pt]
		$5$ & $1.1735$ & $10$ & $1.0929$ & $30$ & $1.0325$ \\[3pt]
		\hline
	\end{tabular}
	\caption{$B(k)$ for $1\leq k\leq 30$.}
\end{table}

We see that 
\[
A(k)=1+\frac{1}{2k}+o(k^{-1})
\]
and almost surely, albeit we do not have a proof,
\[
B(k)=1+\frac{1}{k}+o(k^{-1}).
\]
These estimates are improved in the following result.

\begin{theorem}\label{thm2}
	For $2\leq k\leq 10$ we have
	\[
	\lambda^{(k)} > C(k)
	\]
	with $C(k)$ given in Table 2 below.
	\vspace{2mm}
\begin{table}[h!]
	\centering
	\begin{tabular}{ |c|c||c|c||c|c|}
		\hline
		$k$ & $C(k)$ &  $k$ & $C(k)$ & $k$ & $C(k)$ \\
		\hline\hline
		$2$ & $1.606$ & $5$ & $1.306$ & $8$ & $1.205$ \\[3pt]
		$3$ & $1.451$ & $6$ & $1.265$ & $9$ & $1.184$ \\[3pt]
		$4$ & $1.365$ & $7$ & $1.231$ & $10$ & $1.167$ \\[3pt]
		\hline
	\end{tabular}
	\caption{$C(k)$ for $2\leq k\leq 10$.}
\end{table}	
\end{theorem}
\newpage
We may compare these results as follows.
\vspace{2mm}

\begin{table}[h!]
	\centering
	\begin{tabular}{ |c|c|c||c|c|c||c|c|}
		\hline
		$k$ & $\frac{B(k)-1}{A(k)-1}$ & $\frac{C(k)-1}{A(k)-1}$ & $k$ & $\frac{B(k)-1}{A(k)-1}$ & $\frac{C(k)-1}{A(k)-1}$ & $k$ & $\frac{B(k)-1}{A(k)-1}$ \\
		\hline\hline
		$1$ & $1.8773$ & $3.0928$ & $6$ & $1.9944$ & $3.5727$ & $14$ & $1.9988$ \\[3pt]
		$2$ & $1.9702$ & $3.3075$ & $7$ & $1.9958$ & $3.5768$ & $18$ & $1.9993$ \\[3pt]
		$3$ & $1.9818$ & $3.3683$ & $8$ & $1.9967$ & $3.5846$ & $22$ & $1.9995$ \\[3pt]
		$4$ & $1.9888$ & $3.4583$ & $9$ & $1.9973$ & $3.5856$ & $26$ & $1.9996$ \\[3pt]
		$5$ & $1.9924$ & $3.5126$ & $10$ & $1.9978$ & $3.5886$ & $30$ & $1.9997$ \\[3pt]
		\hline
	\end{tabular}
	\caption{$\frac{B(k)-1}{A(k)-1}$ for $1\leq k\leq 30$ and $\frac{C(k)-1}{A(k)-1}$ for $1\leq k\leq 10$.}
\end{table}

It appears from Table 3 that not only
	\[
	\lim_{k\rightarrow\infty}\frac{B(k)-1}{A(k)-1}=2
	\]
but that the convergence is rapid. It would be desirable to have a proof of this result. The value for $C(k)$ is more complicated. Its treatment involves not only the shifted variable in $Z^{(k)}$ but also the introduction of the amplifier $A(s,P)$ (as in Section 2.1 below). The amplifier involves the polynomial $P$ and the parameter $y_1=T^{\vartheta_1}$ in which $\vartheta_1<1/4$. It is possible that this last condition might be relaxed, leading to a better value for $C(k)$, and, ultimately, a value for 
\[
	\lim_{k\rightarrow\infty}\frac{C(k)-1}{A(k)-1}.
	\]

As remarked in [\textbf{\ref{BH}}], when $k$ and $\ell$ have the same parity, the zeros of $Z^{(k)}(t)$ and $Z^{(\ell)}(t)$ seem to come in pairs which are very close to each other. See the below graphs for the zeros of $Z(t)$ and $Z''(t)$, and the zeros of $Z^{(4)}(t)$ and $Z^{(6)}(t)$.

\vspace{0.3cm}

\begin{center}
\includegraphics[width=0.85\textwidth]{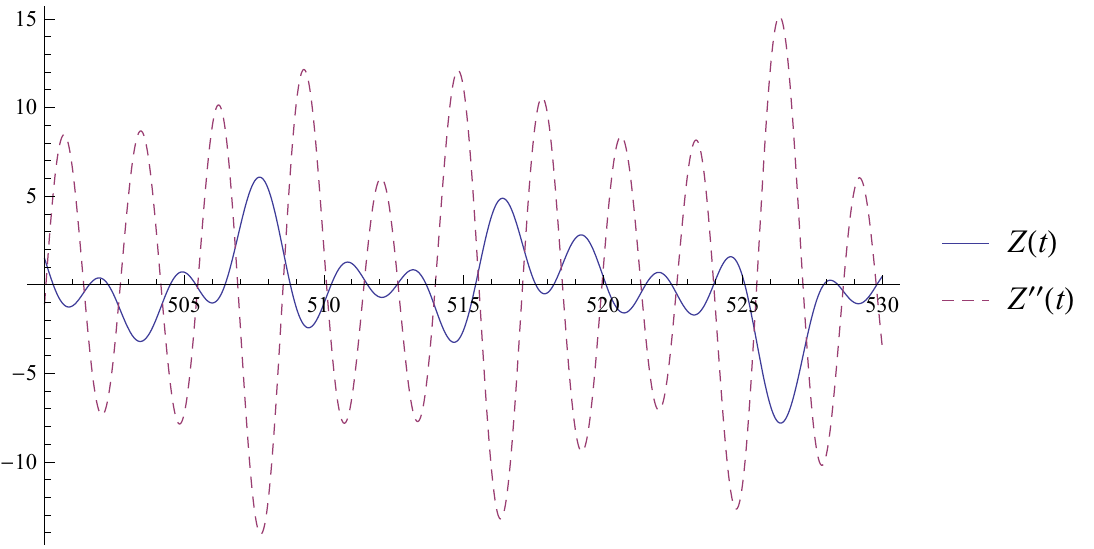}

\vspace{0.2cm}
\textsc{Figure 1}. $Z(t)$ and $Z''(t)$ for $t\in[500,530]$.
\end{center}
\vspace{0.3cm}

\begin{center}
\includegraphics[width=0.85\textwidth]{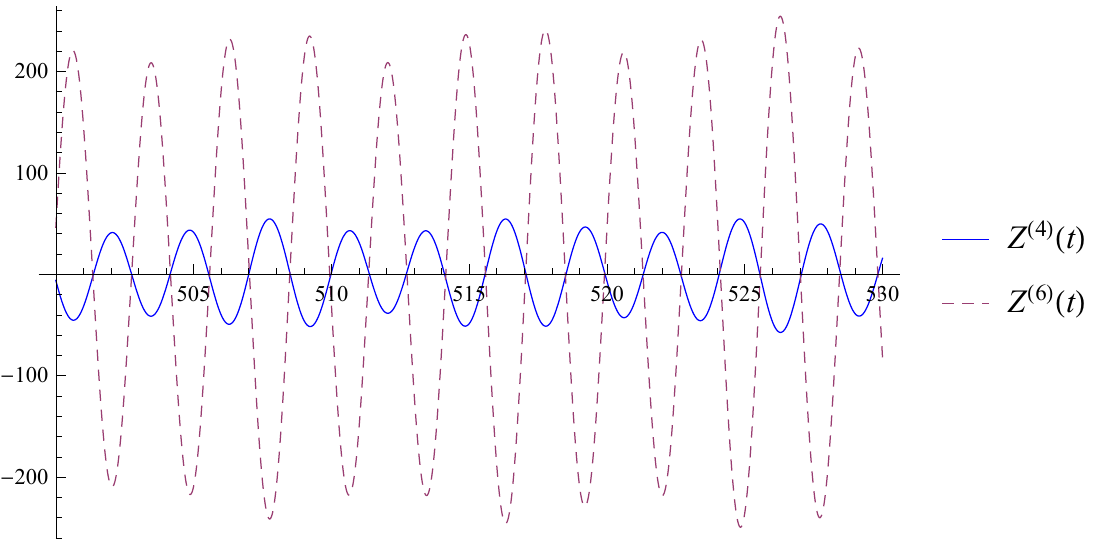}

\vspace{0.2cm}
\textsc{Figure 2}. $Z^{(4)}(t)$ and $Z^{(6)}(t)$ for $t\in[500,530]$.
\end{center}
\vspace{0.3cm}

For $x\in\mathbb{R}$ and a set $S\subset\mathbb{R}$ we define
\[
\text{dist}(x,S):=\inf_{s\in S}|x-s|.
\]
Let
\[
\mu^{(k)}:=\liminf_{n\rightarrow\infty}\frac{\text{dist}(t_n,S^{(k)})}{2\pi/\log t_n}.
\]
Our next theorem shows that there exist infinitely many zeros of $Z(t)$ whose distances to the set of zeros of $Z^{(2k)}(t)$ are small.

\begin{theorem}\label{thmsmall}
Assuming RH, then we have
\[
\mu^{(k)}<\mu_k
\]
with
\begin{equation}\label{muvalues}
\mu_1=0.2,\qquad\mu_2=0.22\qquad\emph{and}\qquad \mu_3=0.23.
\end{equation}
\end{theorem}

Our method does not produce a positive proportion of distances. However, we can show that there are $\gg_\varepsilon T^{1-\varepsilon}$ such small distances.

\begin{corollary}\label{coro}
Let $T$ be large. Assuming RH, then for $1\leq k\leq 3$ and any $C>\log 4$, there are
\[
\gg T\exp\Big(-C\frac{\log T}{\log\log T}\Big)
\]
ordinates $t_n\in[T,2T]$ whose distances to the set of zeros of $Z^{(2k)}(t)$ are $<\frac{2\pi}{\log T}\mu_k$ with $\mu_k$ given in \eqref{muvalues}.
\end{corollary}

Numerical data seems to show that the distance between $t_n^{(k)}$ and the set of zeros of $Z^{(\ell)}(t)$ is small for almost all $n$ for $k,\ell>0$ with the same parity (see, for instance, Figure 2). It would be interesting if anything can be said about the size of
\[
\mu^{(k,\ell)}:=\liminf_{n\rightarrow\infty}\frac{\text{dist}(t_n^{(k)},S^{(\ell)})}{2\pi/\log t_n^{(k)}}
\]
when $k,\ell>0$ are both odd or both even.

\section{Reduction to mean values of the Riemann zeta-function}

\subsection{Large gaps}

For large gap results, we shall follow the ideas in [\textbf{\ref{H3}}] and [\textbf{\ref{BM}}].

We use the following version of Wirtinger's Inequality, which can be found in Theorem
257 of Hardy, Littlewood and P\'olya [\textbf{\ref{HLP}}].

\begin{theorem}\label{Wi}
If $f\in C^2[a,b]$ and $f(a)=f(b)=0$, then we have
\begin{equation*}
\int_{a}^{b}|f(t)|^2dt\leq\Big(\frac{b-a}{\pi}\Big)^2\int_{a}^{b}|f'(t)|^2dt.
\end{equation*}
\end{theorem}

Suppose that 
\begin{equation}\label{assum}
\lambda^{(k)}\leq \lambda.
\end{equation}
Let $A(s)$ be an ``amplifier" and let
\[
f_k(t;\lambda,v,\eta,A):=e^{2iv\theta(t)}Z^{(k)}(t)Z^{(k)}\Big(t+\frac{\pi\lambda}{\log T}\Big)A\Big(\frac12+it+\frac{2\pi i \eta}{\log T}\Big),
\]
where $v,\eta\in\mathbb{R}$ are to be chosen later. Denote all the zeros of the function $f_k$ in the interval $[T,2T]$ by $v_1\leq v_2\leq\ldots\leq v_N$ . Then from \eqref{assum} we have
\[
v_{n+1}-v_n\leq \big(1+o(1)\big)\frac{\pi\lambda}{\log T}
\]
as $T\rightarrow\infty$ for $0\leq n\leq N$, where we have defined $v_0=T$ and $v_{N+1}=2T$. It follows from Theorem \ref{Wi} that
\begin{eqnarray*}
\int_{v_n}^{v_{n+1}}|f_k(t;\lambda,v,\eta,A)|^2dt\leq\big(1+o(1)\big)\frac{\lambda^2}{(\log T)^2}\int_{v_n}^{v_{n+1}}|f_k'(t;\lambda,v,\eta,A)|^2dt
\end{eqnarray*}
for $0\leq n\leq N$. Summing up we obtain
\begin{eqnarray*}
\int_{T}^{2T}|f_k(t;\lambda,v,\eta,A)|^2dt\leq\big(1+o(1)\big)\frac{\lambda^2}{(\log T)^2}\int_{T}^{2T}|f_k'(t;\lambda,v,\eta,A)|^2dt.
\end{eqnarray*}
Thus, if the ratio
\begin{equation}\label{requiredlarge}
h_k(\lambda,v,\eta,A):=\frac{(\log T)^2}{\lambda^2}\frac{\int_{T}^{2T}|f_k(t;\lambda,v,\eta,A)|^2dt}{\int_{T}^{2T}|f_k'(t;\lambda,v,\eta,A)|^2dt}>1
\end{equation}
for some $\lambda,v, \eta$ and $A$, then we get $$\lambda^{(k)}>\lambda.$$

We shall choose the amplifier of the form
\begin{equation*}
A(s):=A(s,P) = \sum_{n\leq y_1}\frac{P(\frac{\log y_1/n}{\log{y_1}})}{n^s},
\end{equation*}
where $y_1=T^{\vartheta_1}$ with $\vartheta_1<1/4$, which is admissible following the twisted fourth moment of the Riemann zeta-function in [\textbf{\ref{BBLR}}], and 
$P$ is a polynomial to be chosen later\footnote{We could also add a factor of $d_r(n)$, the coefficients of $\zeta(s)^r$, but that does not help with the numerical results so we choose $r=1$ for simplicity.}. 

\subsection{Small gaps}

For small gap results we shall modify the approach in [\textbf{\ref{BGMM}}].

We observe that if 
\begin{equation}\label{assum2}
\mu^{(2k)}> \mu >0,
\end{equation}
 then 
\[
Z^{(2k)}\Big(\gamma+\frac{2\pi \mu}{\log T}\Big)Z^{(2k)}\Big(\gamma-\frac{2\pi \mu}{\log T}\Big)  \Big|M\Big(\frac12+i\gamma\Big)\bigg|^2 \geq0
\]
for any mollifier $M$. That is because $Z^{(2k)}(\gamma+\frac{2\pi \mu}{\log T})$ and $Z^{(2k)}(\gamma-\frac{2\pi \mu}{\log T})$ are of the same sign when $T< \gamma\leq 2T$ sufficiently large by \eqref{assum2}. Therefore, as $Z^{(2k)}(t)=Z^{(2k)}(-t)$, if we can show that 
\begin{equation}\label{requiredsmall}
\Sigma_k(\mu,M) :=\sum_{T<\gamma \le 2T}Z^{(2k)}\Big(\gamma+\frac{2\pi \mu}{\log T}\Big)Z^{(2k)}\Big(-\gamma+\frac{2\pi \mu}{\log T}\Big)  \Big|M\Big(\frac12+i\gamma\Big)\bigg|^2 < 0
\end{equation}
for some $\mu$ and $M$, then it follows that $$\mu^{(2k)} \leq \mu.$$ 

The mollifier $M$ is chosen to be the usual mollifier,
\begin{equation*}
M(s):=M(s,P) = \sum_{n\leq y_2}\frac{\mu(n)P(\frac{\log y_2/n}{\log{y_2}})}{n^s},
\end{equation*}
where $y_2=T^{\vartheta_2}$ with $\vartheta_2<1/2$, which is admissible following the mollified discrete second moment of $\zeta'(\rho)$ in [\textbf{\ref{CGG1}}] and [\textbf{\ref{BH-B}}], and $P$ is a polynomial with $P(0)=0$.

\section{Proof of Theorem \ref{thm1} and Theorem \ref{thm2}}

\subsection{The mean value results}

Let
\begin{eqnarray}\label{Iab}
I_{\underline{\alpha},\underline{\beta}}(A)&=&\int_{T}^{2T}\zeta(\tfrac{1}{2}+\alpha_1+it)\zeta(\tfrac{1}{2}+\alpha_2+it)\zeta(\tfrac{1}{2}+\beta_1-it)\zeta(\tfrac{1}{2}+\beta_2-it)\\
&&\qquad\qquad\times\, A(\tfrac{1}{2}+\alpha_3+it)A(\tfrac{1}{2}+\beta_3-it)dt.\nonumber
\end{eqnarray}
We need the following result from [\textbf{\ref{BM}}] (see Lemma 3.1).

\begin{theorem}\label{main}
We have
\[
I_{\underline{\alpha},\underline{\beta}}(A)=\frac{Cc(\underline{\alpha},\underline{\beta})}{2}\,T(\log y_1)^{5}(\log T)^4+O(T(\log T)^{8})
\]
uniformly for $\alpha_j,\beta_j\ll(\log T)^{-1}$, where 
\[
C=\prod_{p}\bigg(\Big(1-\frac1p\Big)^{9}\sum_{j\geq0}\frac{d_3(p^j)^2}{p^j}\bigg)
\]
and $c(\underline{\alpha},\underline{\beta})$ is given by
\begin{align}\label{ca}
&\mathop{\int}_{\substack{[0,1]^9\\x+x_1+x_2\leq1\\x+x_3+x_4\leq 1}}y_1^{-(\alpha_3+\beta_3)x-\alpha_3( x_1+ x_2)-\beta_3(x_3+x_4)-\beta_1 x_1-\beta_2 x_2-\alpha_1x_3-\alpha_2x_4-(\alpha_2-\alpha_1)(x_3-x_4)t_3-(\beta_2-\beta_1)(x_1-x_2)t_4}\nonumber\\
&\quad\times(Ty_1^{-x_1-x_3})^{-(\alpha_1+\beta_1)t_1-(\alpha_2-\alpha_1)t_1t_3-(\beta_2-\beta_1)t_1t_4}(Ty_1^{-x_2-x_4})^{-(\alpha_2+\beta_2)t_2+(\alpha_2-\alpha_1)t_2t_3+(\beta_2-\beta_1)t_2t_4}\nonumber\\
&\quad\times \big(1-\vartheta_1(x_1+x_3)\big)\big(1-\vartheta_1(x_2+x_4)\big)\nonumber\\
&\quad\times\Big(\vartheta_1(x_1-x_2)+\big(1-\vartheta_1(x_1+x_3)\big)t_1-\big(1-\vartheta_1(x_2+x_4)\big)t_2\Big)\nonumber\\
&\quad\times\Big(\vartheta_1(x_3-x_4)+\big(1-\vartheta_1(x_1+x_3)\big)t_1-\big(1-\vartheta_1(x_2+x_4)\big)t_2\Big) \nonumber\\
&\quad\times P(1-x-x_1-x_2)P(1-x-x_3-x_4)dx_1dx_2dx_3dx_4dxdt_1dt_2dt_3dt_4.
\end{align}

\end{theorem}

Using Theorem \ref{main} we will prove the following proposition.

\begin{proposition}\label{main1}
We have
\begin{equation*}
\int_{T}^{2T}|f_k^{(j)}(t;\lambda,v,\eta,A)|^2dt=\frac{Cc_{k,j}(\lambda,v,\eta)}{2}T(\log y_1)^{5}(\log T)^{8+2j}+O(T(\log T)^{12+2j})
\end{equation*}
for $j=0,1$, where $c_{k,j}(\lambda,v,\eta)$ is given by
\begin{align*}
&\mathop{\int}_{\substack{[0,1]^9\\x+x_1+x_2\leq1\\x+x_3+x_4\leq 1}}e^{i\vartheta_1\lambda\pi\big(x_2-x_4-(x_3-x_4)t_3+(x_1-x_2)t_4\big)-i\lambda\pi\big((1-\vartheta_1(x_1+x_3))t_1-(1-\vartheta_1(x_2+x_4))t_2\big)(t_3-t_4)}\\
&\quad\times \big(1-\vartheta_1(x_1+x_3)\big)\big(1-\vartheta_1(x_2+x_4)\big)\\
&\quad\times\Big(\vartheta_1(x_1-x_2)+\big(1-\vartheta_1(x_1+x_3)\big)t_1-\big(1-\vartheta_1(x_2+x_4)\big)t_2\Big)\\
&\quad\times\Big(\vartheta_1(x_3-x_4)+\big(1-\vartheta_1(x_1+x_3)\big)t_1-\big(1-\vartheta_1(x_2+x_4)\big)t_2\Big)\nonumber\\
&\quad\times \Big(\frac12-\vartheta_1 x_3+\vartheta_1(x_3-x_4)t_3-\big(1-\vartheta_1(x_1+x_3)\big)t_1(1-t_3)-\big(1-\vartheta_1(x_2+x_4)\big)t_2t_3\Big)^k\\
&\quad\times \Big(\frac12-\vartheta_1 x_4-\vartheta_1(x_3-x_4)t_3-\big(1-\vartheta_1(x_1+x_3)\big)t_1t_3-\big(1-\vartheta_1(x_2+x_4)\big)t_2(1-t_3)\Big)^k\\
&\quad\times \Big(\frac12-\vartheta_1 x_1+\vartheta_1(x_1-x_2)t_4-\big(1-\vartheta_1(x_1+x_3)\big)t_1(1-t_4)-\big(1-\vartheta_1(x_2+x_4)\big)t_2t_4\Big)^k\\
&\quad\times \Big(\frac12-\vartheta_1 x_2-\vartheta_1(x_1-x_2)t_4-\big(1-\vartheta_1(x_1+x_3)\big)t_1t_4-\big(1-\vartheta_1(x_2+x_4)\big)t_2(1-t_4)\Big)^k\\
&\quad\times \Big(v+1-\vartheta_1(x+x_1+x_2+x_3+x_4)-\big(1-\vartheta_1(x_1+x_3)\big)t_1-\big(1-\vartheta_1(x_2+x_4)\big)t_2\Big)^{2j}\\
&\quad\times P(1-x-x_1-x_2)P(1-x-x_3-x_4)dx_1dx_2dx_3dx_4dxdt_1dt_2dt_3dt_4.
\end{align*}
\end{proposition}
\begin{proof}
Note that
\begin{equation}\label{thetaprop}
\theta'(t)=\frac{\log T}{2}+O(1)\qquad\text{and}\qquad \theta^{(k)}(t)\ll_k T^{-k+1}
\end{equation}
for $|t|\asymp T$ and any $k\geq 2$. So
\begin{align}\label{Z'tformula}
&Z^{(k)}(t+a)\nonumber\\
&\quad=i^ke^{i\theta(t+a)}\sum_{j=0}^{k}\binom{k}{j}\Big(\frac{\log T}{2}\Big)^j\zeta^{k-j}(\tfrac12+it+ia)+O\bigg(\sum_{j=0}^{k}(\log T)^{j-1}|\zeta^{k-j}(\tfrac12+it+ia)|\bigg)\nonumber\\
&\quad=i^ke^{i\theta(t+a)}(\log T)^kQ\Big(\frac{1}{\log T}\frac{\partial}{\partial\alpha}\Big)\zeta(\tfrac{1}{2}+\alpha+it)\Big|_{\alpha=ia}\\
&\qquad\qquad\qquad+O\bigg(\sum_{j=0}^{k}(\log T)^{j-1}|\zeta^{k-j}(\tfrac12+it+ia)|\bigg),\nonumber
\end{align}
where
\[
Q(x)=\Big(\frac12+x\Big)^k.
\]
By Theorem \ref{main}, the contribution of the $O$-term to $\int|f_k(t;\lambda,v,\eta,A)|^2$ is $O(T(\log T)^{4k+8})$.
Hence
\begin{align}\label{abc}
&\int_{T}^{2T}|f_k(t;\lambda,v,\eta,A)|^2dt\nonumber\\
&\qquad=(\log T)^{4k}Q\Big(\frac{1}{\log T}\frac{\partial}{\partial\alpha_1}\Big)Q\Big(\frac{1}{\log T}\frac{\partial}{\partial\alpha_2}\Big)Q\Big(\frac{1}{\log T}\frac{\partial}{\partial\beta_1}\Big)Q\Big(\frac{1}{\log T}\frac{\partial}{\partial\beta_2}\Big)\\
&\qquad\qquad\qquad I_{\underline{\alpha},\underline{\beta}}(A)\bigg|_{\substack{\alpha_1=\beta_1=0\\\alpha_2=\pi i\lambda/\log T,\beta_2=-\pi i\lambda/\log T\\\alpha_3=2\pi i\eta/\log T,\beta_3=-2\pi i\eta/\log T}}+O(T(\log T)^{4k+8}),\nonumber
\end{align}
where $I_{\underline{\alpha},\underline{\beta}}(A)$ is given in \eqref{Iab}.

Similarly, ignoring various error terms we have
\begin{align*}
&f_k(t;\lambda,v,\eta,A)=i^k(\log T)^{2k}Q\Big(\frac{1}{\log T}\frac{\partial}{\partial\alpha_1}\Big)Q\Big(\frac{1}{\log T}\frac{\partial}{\partial\alpha_2}\Big)\\
&\qquad\qquad e^{i(2v+1)\theta(t)+i\theta(t+\lambda\pi/\log T)}\zeta(\tfrac12+\alpha_1+it)\zeta(\tfrac12+\alpha_2+it)A(\tfrac12+\alpha_3+it)\bigg|_{\substack{\alpha_1=0\\ \alpha_2=\pi i\lambda/\log T\\\alpha_3=2\pi i\eta/\log T}}
\end{align*}
and
\begin{align*}
&f_k'(t;\lambda,v,\eta,A)\\
&\qquad=i^{k+1}(\log T)^{2k+1}Q\Big(\frac{1}{\log T}\frac{\partial}{\partial\alpha_1}\Big)Q\Big(\frac{1}{\log T}\frac{\partial}{\partial\alpha_2}\Big)R\bigg(\frac{1}{\log T}\Big(\frac{\partial}{\partial\alpha_1}+\frac{\partial}{\partial\alpha_2}+\frac{\partial}{\partial\alpha_3}\Big)\bigg)\\
&\qquad\qquad e^{i(2v+1)\theta(t)+i\theta(t+\lambda\pi/\log T)}\zeta(\tfrac12+\alpha_1+it)\zeta(\tfrac12+\alpha_2+it)A(\tfrac12+\alpha_3+it)\bigg|_{\substack{\alpha_1=\alpha_3=0\\ \alpha_2=\pi i\lambda/\log T\\\alpha_3=2\pi i\eta/\log T}},
\end{align*}
where
\[
R(x)=v+1+x.
\]
So
\begin{align}\label{abcd}
&\int_{-\infty}^{\infty}|f_k'(t;\lambda,v,\eta,A)|^2w\Big(\frac tT\Big)dt\nonumber\\
&\qquad=(\log T)^{4k+2}Q\Big(\frac{1}{\log T}\frac{\partial}{\partial\alpha_1}\Big)Q\Big(\frac{1}{\log T}\frac{\partial}{\partial\alpha_2}\Big)Q\Big(\frac{1}{\log T}\frac{\partial}{\partial\beta_1}\Big)Q\Big(\frac{1}{\log T}\frac{\partial}{\partial\beta_2}\Big)\nonumber\\
&\qquad\qquad R\bigg(\frac{1}{\log T}\Big(\frac{\partial}{\partial\alpha_1}+\frac{\partial}{\partial\alpha_2}+\frac{\partial}{\partial\alpha_3}\Big)\bigg)R\bigg(\frac{1}{\log T}\Big(\frac{\partial}{\partial\beta_1}+\frac{\partial}{\partial\beta_2}+\frac{\partial}{\partial\beta_3}\Big)\bigg)\nonumber\\
&\qquad\qquad\qquad I_{\underline{\alpha},\underline{\beta}}(A)\bigg|_{\substack{\alpha_1=\beta_1=0\\\alpha_2=\pi i\lambda/\log T,\beta_2=-\pi i\lambda/\log T\\\alpha_3=2\pi i\eta/\log T,\beta_3=-2\pi i\eta/\log T}}+O(T(\log T)^{4k+10}).
\end{align}

Since $I_{\underline{\alpha},\underline{\beta}}$ and $c(\underline{\alpha},\underline{\beta})$ are holomorphic with respect to $\alpha_j, \beta_j$ small, the derivatives appearing in \eqref{abc} and \eqref{abcd} can be obtained as integrals of radii $\asymp (\log T)^{-1}$ around the points $\alpha_1=\beta_1=0,\alpha_2=\pi i\lambda/\log T,\beta_2=-\pi i\lambda/\log T$ and $\alpha_3=2\pi i\eta/\log T,\beta_3=-2\pi i\eta/\log T$ using Cauchy's residue theorem. So we can get $c_{k,0}(\lambda,v,\eta)$ and $c_{k,1}(\lambda,v,\eta)$ by applying the differential operators in \eqref{abc} and \eqref{abcd} to $c(\underline{\alpha},\underline{\beta})$.  

Next we check that applying the above differential operators to $c(\underline{\alpha},\underline{\beta})$ does indeed give $c_{k,0}(\lambda,v,\eta)$ and $c_{k,1}(\lambda,v,\eta)$.
Note that
\begin{equation}\label{eq:Qop}
Q\Big(\frac{1}{\log T}\frac{\partial}{\partial\alpha}\Big)X^{\alpha}=Q\Big(\frac{\log X}{\log T}\Big)X^{\alpha}
\end{equation}
and
\begin{equation}
\label{eq:Rop}
R\bigg(\frac{1}{\log T}\Big(\frac{\partial}{\partial\alpha_1}+\frac{\partial}{\partial\alpha_2}+\frac{\partial}{\partial\alpha_3}\Big)\bigg)X_{1}^{\alpha_1}X_{2}^{\alpha_2}X_{3}^{\alpha_3}=R\Big(\frac{\log (X_1X_2X_3)}{\log T}\Big)X_{1}^{\alpha_1}X_{2}^{\alpha_2}X_{3}^{\alpha_3}.
\end{equation}
Using \eqref{eq:Qop}, \eqref{eq:Rop} and \eqref{ca}, we have
\begin{align*}
&Q\Big(\frac{1}{\log T}\frac{\partial}{\partial\alpha_1}\Big)Q\Big(\frac{1}{\log T}\frac{\partial}{\partial\alpha_2}\Big)Q\Big(\frac{1}{\log T}\frac{\partial}{\partial\beta_1}\Big)Q\Big(\frac{1}{\log T}\frac{\partial}{\partial\beta_2}\Big)\\
&\qquad R\bigg(\frac{1}{\log T}\Big(\frac{\partial}{\partial\alpha_1}+\frac{\partial}{\partial\alpha_2}+\frac{\partial}{\partial\alpha_3}\Big)\bigg)^jR\bigg(\frac{1}{\log T}\Big(\frac{\partial}{\partial\beta_1}+\frac{\partial}{\partial\beta_2}+\frac{\partial}{\partial\beta_3}\Big)\bigg)^jc(\underline{\alpha},\underline{\beta})\\
&=\mathop{\int}_{\substack{[0,1]^9\\x+x_1+x_2\leq1\\x+x_3+x_4\leq 1}}y_1^{-(\alpha_3+\beta_3)x-\alpha_3( x_1+ x_2)-\beta_3(x_3+x_4)-\beta_1 x_1-\beta_2 x_2-\alpha_1x_3-\alpha_2x_4-(\alpha_2-\alpha_1)(x_3-x_4)t_3-(\beta_2-\beta_1)(x_1-x_2)t_4}\nonumber\\
&\quad\times(Ty_1^{-x_1-x_3})^{-(\alpha_1+\beta_1)t_1-(\alpha_2-\alpha_1)t_1t_3-(\beta_2-\beta_1)t_1t_4}(Ty_1^{-x_2-x_4})^{-(\alpha_2+\beta_2)t_2+(\alpha_2-\alpha_1)t_2t_3+(\beta_2-\beta_1)t_2t_4}\nonumber\\
&\quad\times \big(1-\vartheta_1(x_1+x_3)\big)\big(1-\vartheta_1(x_2+x_4)\big)\nonumber\\
&\quad\times\Big(\vartheta_1(x_1-x_2)+\big(1-\vartheta_1(x_1+x_3)\big)t_1-\big(1-\vartheta_1(x_2+x_4)\big)t_2\Big)\nonumber\\
&\quad\times\Big(\vartheta_1(x_3-x_4)+\big(1-\vartheta_1(x_1+x_3)\big)t_1-\big(1-\vartheta_1(x_2+x_4)\big)t_2\Big) \nonumber\\
&\quad\times Q\Big(-\vartheta_1 x_3+\vartheta_1(x_3-x_4)t_3-\big(1-\vartheta_1(x_1+x_3)\big)t_1(1-t_3)-\big(1-\vartheta_1(x_2+x_4)\big)t_2t_3\Big)\\
&\quad\times Q\Big(-\vartheta_1 x_4-\vartheta_1(x_3-x_4)t_3-\big(1-\vartheta_1(x_1+x_3)\big)t_1t_3-\big(1-\vartheta_1(x_2+x_4)\big)t_2(1-t_3)\Big)\\
&\quad\times Q\Big(-\vartheta_1 x_1+\vartheta_1(x_1-x_2)t_4-\big(1-\vartheta_1(x_1+x_3)\big)t_1(1-t_4)-\big(1-\vartheta_1(x_2+x_4)\big)t_2t_4\Big)\\
&\quad\times Q\Big(-\vartheta_1 x_2-\vartheta_1(x_1-x_2)t_4-\big(1-\vartheta_1(x_1+x_3)\big)t_1t_4-\big(1-\vartheta_1(x_2+x_4)\big)t_2(1-t_4)\Big)\\
&\quad\times R\Big(-\vartheta_1(x+x_1+x_2+x_3+x_4)-\big(1-\vartheta_1(x_1+x_3)\big)t_1-\big(1-\vartheta_1(x_2+x_4)\big)t_2\Big)^{2j}\\
&\quad\times P(1-x-x_1-x_2)P(1-x-x_3-x_4)dx_1dx_2dx_3dx_4dxdt_1dt_2dt_3dt_4.
\end{align*}
Setting $\alpha_1=\beta_1=0,\alpha_2=\pi i\lambda/\log T,\beta_2=-\pi i\lambda/\log T$, $\alpha_3=2\pi i\eta/\log T,\beta_3=-2\pi i\eta/\log T$ and simplifying 
gives Proposition \ref{main1}.
\end{proof}

\subsection{Numerical evaluations}

It follows directly from Proposition \ref{main1} that
\[
h_k(\lambda,v,\eta,A)=\frac{c_{k,0}(\lambda,v,\eta)}{\lambda^2c_{k,1}(\lambda,v,\eta)}+o(1).
\]

With $\vartheta_1=0.2499$ and $P(x)=1-2.5x$, we have
\[
h_1(1.9,0.2,0.5,A)>1.02,
\]
and Theorem \ref{thm1} follows by \eqref{requiredlarge}.

Theorem \ref{thm2} is obtained with $\vartheta_1=0.2499$, $v=0.2$, $\eta=0.5$ and the following choices of polynomial $P$.
\begin{table}[h!]
	\centering
	\begin{tabular}{|c|c||c|c||c|c|}
		\hline
		$P(x)$ &  & $P(x)$ & &$P(x)$&   \\
		\hline\hline
		 $1-2x$& $h_2(1.606)>1.003$  & $1-1.5x$ &  $h_5(1.306)>1.001$ & $1-x$ & $h_8(1.205)>1.001$ \\[3pt]
		$1-1.8x$ & $h_3(1.451)>0.001$ & $1-1.4x$  & $h_6(1.265)>1.005$ &  $1-1.2x$      & $h_9(1.184)>1.01$ \\[3pt]
		 $1-1.7x$& $h_4(1.365)>1.01$ & $1-1.5x$ & $h_7(1.232)>1.008$  & $1-x$ & $h_{10}(1.167)>1.009$ \\[3pt]
		\hline
	\end{tabular}
	\caption{Values of $h_k(\lambda):=h_k(\lambda,0.2,0.5,A)$ for $2\leq k\leq 10$.}
\end{table}

\section{Proof of Theorem \ref{thmsmall}}

\subsection{The mean value result}

We recall from \eqref{Z'tformula} that
\begin{align*}
Z^{(2k)}\Big(t+\frac{a}{\log T}\Big)&=(-1)^ke^{i\theta(t+\frac{a}{\log T})}(\log T)^{2k}Q\Big(\frac{\partial}{\partial\alpha}\Big)\zeta\Big(\tfrac{1}{2}+it+\frac{\alpha}{\log T}\Big)\Big|_{\alpha=ia}\\
&\qquad\qquad+O\bigg(\sum_{j=0}^{2k}(\log T)^{j-1}\Big|\zeta^{(2k-j)}\Big(\tfrac12+it+\frac{ia}{\log T}\Big)\Big|\bigg),\nonumber
\end{align*}
where
\[
Q(x)=\Big(\frac12+x\Big)^{2k}.
\]
Thus,
\begin{align}\label{formula101}
&Z^{(2k)}\Big(\gamma+\frac{2\pi \mu}{\log T}\Big)Z^{(2k)}\Big(-\gamma+\frac{2\pi \mu}{\log T}\Big) \nonumber\\
&\ \ =e^{i(\theta(\gamma+\frac{2\pi \mu}{\log T})+\theta(-\gamma+\frac{2\pi \mu}{\log T}))}(\log T)^{4k}Q\Big(\frac{\partial}{\partial\alpha}\Big)Q\Big(\frac{\partial}{\partial\beta}\Big)\zeta\Big(\rho+\frac{\alpha}{\log T}\Big)\zeta\Big(1-\rho+\frac{\beta}{\log T}\Big)\Big|_{\alpha=\beta=2\pi i \mu}\nonumber\\
&\qquad\quad+O\bigg(\sum_{j_1,j_2\leq 2k}(\log T)^{j_1+j_2-1}\Big|\zeta^{(2k-j_1)}\Big(\rho+\frac{2\pi i \mu}{\log T}\Big)\zeta^{(2k-j_2)}\Big(1-\rho+\frac{2\pi i \mu}{\log T}\Big)\Big|\bigg)\nonumber\\
&\ \ =e^{2\pi i\mu}(\log T)^{4k}Q\Big(\frac{\partial}{\partial\alpha}\Big)Q\Big(\frac{\partial}{\partial\beta}\Big)\zeta\Big(\rho+\frac{\alpha}{\log T}\Big)\zeta\Big(1-\rho+\frac{\beta}{\log T}\Big)\Big|_{\alpha=\beta=2\pi i \mu}\nonumber\\
&\qquad\quad+O\bigg(\sum_{j_1,j_2\leq 2k}(\log T)^{j_1+j_2-1}\Big|\zeta^{(2k-j_1)}\Big(\rho+\frac{2\pi i \mu}{\log T}\Big)\zeta^{(2k-j_2)}\Big(1-\rho+\frac{2\pi i \mu}{\log T}\Big)\Big|\bigg),
\end{align}
by \eqref{thetaprop}.

We next quote a discrete mean value of the Riemann zeta-function. Let $P,Q$ be polynomials with $P(0)=0$ and let
\begin{align*}
	&J(a,b,P,Q)\\
	&\qquad=\sum_{T<\gamma \leq 2T}Q\Big(\frac{\partial}{\partial a}\Big)Q\Big(\frac{\partial}{\partial b}\Big) \zeta\Big(\rho+\frac{a}{\log T}\Big) \, \zeta\Big(1-\rho+\frac{b}{\log T}\Big)\,M(\rho,P) \, M(1-\rho,P).
\end{align*}
Then the following estimate holds.

\begin{theorem}\label{mainthm2}
	If $\vartheta_2<1/2$ and $a,b\ll 1$, then  we have
	\begin{align*}
		&J(a,b,P,Q)\sim\frac{T\log T}{2\pi}\frac{\partial^2}{\partial u \, \partial v}\Bigg\{\bigg(\frac{1}{\vartheta_2}\int_{0}^{1}P_u\, P_v\, dx+\int_{0}^{1}P_u\, dx\int_{0}^{1}P_v\, dx\bigg)\\
		&\qquad\qquad\qquad\qquad\qquad\qquad\qquad\qquad\times\bigg(\int_{0}^{1}T_uQ\, T_vQ\, dx-\int_{0}^{1}T_uQ\, dx\int_{0}^{1}T_vQ\, dx\bigg)\\
		&\qquad\qquad\qquad+\int_{0}^{1}P_u\, dx\int_{0}^{1}P_v\, dx\bigg(Q(0)-\int_{0}^{1}T_uQ\, dx\bigg)\bigg(Q(0)-\int_{0}^{1}T_vQ\, dx\bigg)\Bigg\}\Bigg|_{u=v=0}
	\end{align*}
	as $T\rightarrow\infty$, where $P_u=P(x+u)$, $P_v=P(x+v)$,
	\[
	T_uQ=e^{-a(x+\vartheta_2 u)}Q(-x-\vartheta_2 u) \qquad\text{and}\qquad T_vQ=e^{-b(x+\vartheta_2 v)}Q(-x-\vartheta_2 v).
	\]
\end{theorem}

Conrey, Ghosh and Gonek stated this theorem in [\textbf{\ref{CGG87}}; Theorem 2] but did not provide a proof. A conditional proof assuming RH and the Generalised Lindel\"of Hypothesis was later given for the first derivative of the zeta-function without the shifts in [\textbf{\ref{CGG1}}]. The latter assumption was successfully removed by Bui and Heath-Brown [\textbf{\ref{BH-B}}]. Using the ideas in [\textbf{\ref{CGG1}}] and [\textbf{\ref{BH-B}}], Heap, Li and Zhao [\textbf{\ref{HLZ}}; Theorem 2] recently gave a proof with the shifts and even more general coefficients.

By Cauchy-Schwarz's inequality, Theorem \ref{mainthm2} and Cauchy's residue theorem, the contribution of the $O$-term in \eqref{formula101} to $\Sigma(\mu,M)$ is
\begin{align*}
&\ll \sum_{j_1,j_2\leq 2k}(\log T)^{j_1+j_2-1}\bigg(\sum_{T<\gamma \leq 2T}\Big|\zeta^{(2k-j_1)}\Big(\rho+\frac{2\pi i \mu}{\log T}\Big)M(\rho)\Big|^2\bigg)^{1/2}	\\
&\qquad\qquad\times \bigg(\sum_{T<\gamma \leq 2T}\Big|\zeta^{(2k-j_2)}\Big(1-\rho+\frac{2\pi i \mu}{\log T}\Big)M(\rho)\Big|^2\bigg)^{1/2}\\
&\ll T(\log T)^{4k}.
\end{align*}
Thus
\[
\Sigma_k(\mu,M)=e^{2\pi i\mu}(\log T)^{4k}J(2\pi i\mu,2\pi i\mu,P,Q)+O\big(T(\log T)^{4k}).
\]

We choose $\vartheta_1=0.499$, $P(x)=x^2$ and obtain
\[
\Sigma_k(\mu_k,M)=(-r_k+o(1))\frac{T(\log T)^{4k+1}}{2\pi}
\]
with
\[
\mu_1=0.2,\ r_1=0.0018;\qquad \mu_2=0.22,\ r_2=0.0002;\qquad \text{and}\qquad \mu_3=0.23,\ r_3=0.00002.
\] 
This proves Theorem \ref{thmsmall} by \eqref{requiredsmall}.

\section{Proof of Corollary \ref{coro}}

This is a modification of the idea of Bui, Goldston, Milinovich and Montgomery in Section 6 of [\textbf{\ref{BGMM}}].

Let $T$ be large and let $\mu$ and $P$ be chosen as in the proof of Theorem \ref{thmsmall}. Then 
\[
\sum_{T<\gamma \le 2T}Z^{(2k)}\Big(-\gamma+\frac{2\pi \mu}{\log T}\Big) Z^{(2k)}\Big(\gamma+\frac{2\pi \mu}{\log T}\Big) \Big|M\Big(\frac12+i\gamma\Big)\Big|^2 < 0.
\]
 Moreover, Theorem \ref{thmsmall} shows that this sum is $\gg T (\log T)^{4k+1}$ in magnitude. Note that every negative term in this sum corresponds to an ordinate $t_n\in[T,2T]$ with $Z^{(2k)}(t_n-\frac{2\pi\mu}{\log T})Z^{(2k)}(t_n+\frac{2\pi\mu}{\log T}) <0$. Therefore, 
\[
\begin{split}
T(\log T)^{4k+1} &\ll \bigg| \sum_{T<\gamma \le 2T} Z^{(2k)}\Big(-\gamma+\frac{2\pi \mu}{\log T}\Big) Z^{(2k)}\Big(\gamma+\frac{2\pi \mu}{\log T}\Big) \Big|M\Big(\frac12+i\gamma\Big)\Big|^2 \bigg| 
\\
&\le 
\bigg| \!\!\!\!\!\!\!\!\sum_{\substack{T<\gamma \le 2T \\ Z^{(2k)}(\gamma-\frac{2\pi\mu}{\log T})\\\ \ \ \ \ \times Z^{(2k)}(\gamma+\frac{2\pi\mu}{\log T}) <0  }}Z^{(2k)}\Big(-\gamma+\frac{2\pi \mu}{\log T}\Big) Z^{(2k)}\Big(\gamma+\frac{2\pi \mu}{\log T}\Big) \Big|M\Big(\frac12+i\gamma\Big)\Big|^2 \bigg|.
\end{split}
\]
An application of Cauchy-Schwarz's inequality to the latter sum then leads to
\begin{align}\label{finalestimate}
T^2(\log T)^{8k+2} &\le \bigg( \!\!\!\!\!\!\!\!\sum_{\substack{T<\gamma \le 2T \\ Z^{(2k)}(\gamma-\frac{2\pi\mu}{\log T})\\\ \ \ \ \ \times Z^{(2k)}(\gamma+\frac{2\pi\mu}{\log T}) <0  }} 1 \bigg) \\
&\qquad\times \bigg( \sum_{T<\gamma \le 2T} Z^{(2k)}\Big(-\gamma+\frac{2\pi \mu}{\log T}\Big)^2 Z^{(2k)}\Big(\gamma+\frac{2\pi \mu}{\log T}\Big)^2 \Big|M\Big(\frac12+i\gamma\Big)\Big|^4  \bigg),\nonumber
\end{align}
where we have extended the last sum to all $T<\gamma\le 2T$, by positivity.

 By \eqref{Z'tformula} and the upper bound for $|\zeta(\frac{1}{2}+\alpha+it)|$ in [\textbf{\ref{CS}}, \textbf{\ref{CC}}]  we have
\[
Z^{(2k)}(t+a) \ll \exp\!\bigg( \Big(\frac{\log 2}{2} + o(1)\Big) \frac{\log T}{\log\log T} \bigg)
\]
for $t\in[T,2T]$ and $a\ll (\log T)^{-1}$. From [\textbf{\ref{BGMM}}; p.15] we also have
\[
\begin{split}
\sum_{T<\gamma\le 2T} \bigg| M\Big( \frac{1}{2}+i\gamma\Big) \bigg|^4 &\ll T (\log T)^5.
\end{split}
\]
 Hence
\begin{align*}
&\sum_{T<\gamma \le 2T} Z^{(2k)}\Big(-\gamma+\frac{2\pi \mu}{\log T}\Big)^2 Z^{(2k)}\Big(\gamma+\frac{2\pi \mu}{\log T}\Big)^2 \Big|M\Big(\frac12+i\gamma\Big)\Big|^4 \\
&\qquad\qquad\ll T \, \exp\!\bigg( \big(\log 4 + o(1)\big) \frac{\log T}{\log\log T} \bigg).
\end{align*}
Combining with \eqref{finalestimate} we obtain that
\[
\!\!\!\!\!\!\!\!\sum_{\substack{T<\gamma \le 2T \\ Z^{(2k)}(\gamma-\frac{2\pi\mu}{\log T})\\\ \ \ \ \ \times Z^{(2k)}(\gamma+\frac{2\pi\mu}{\log T}) <0  }} 1 \ \gg \ T \, \exp\!\bigg( -\big(\log 4 + o(1)\big) \frac{\log T}{\log\log T} \bigg),
\]
and the corollary follows.

\noindent\textbf{Acknowledgements.}\ \ We would like to thank Micah Milinovich for many helpful discussions..

\end{document}